\documentclass[10pt]{amsart}
\usepackage{latexsym, amsmath, amsfonts, amssymb, mathrsfs, fancyhdr, tikz, tikz-cd,accents}
\usepackage{verbatim}
\usepackage{multicol}
\usetikzlibrary{matrix,arrows,decorations.pathmorphing,calc}

\makeatletter
\@namedef{subjclassname@2020}{%
  \textup{2020} Mathematics Subject Classification}
\makeatother

\def\ds{\displaystyle}
\def\nin{\not \in}

\def\R{\mathbb{R}}

\def\e{\varepsilon}
\def\s{\sigma}

\def\l{\lambda}

\def\S{\mathcal{S}}

\def\a{\alpha}
\def\b{\beta}

\def\t{\tau}

\def\k{\alpha}

\def\ch2{\mathbb{C} \mathbb{H}^2}
\def\h2{\mathbb{H}^2}

\def\ddt{\frac{\partial}{\partial \theta}}

\def\th{\theta}

\def\oshr/2{\cosh \left( \frac{r}{2} \right)}
\def\inhr/2{\sinh \left( \frac{r}{2} \right)}
\def\osh2r/2{\cosh^2 \left( \frac{r}{2} \right)}
\def\inh2r/2{\sinh^2 \left( \frac{r}{2} \right)}
\def\-1/4{- \frac{1}{4}}
\def\H{\mathbb{H}}
\def\C{\mathbb{C}}
\def\S{\mathbb{S}}

\def\y{\mathfrak{h}}
\def\cm{\mathfrak{c}}

\def\Sym{\text{Met} \left( X_k \right)}

\newtheorem{theorem}{Theorem}[section]
\newtheorem{lemma}[theorem]{Lemma}
\newtheorem{cor}[theorem]{Corollary}

\newtheorem{proposition}[theorem]{Proposition}

\theoremstyle{definition}

\theoremstyle{remark}
\newtheorem{remark}[theorem]{Remark}

\numberwithin{equation}{section}

\begin{document}

\title[Explicit description of K\"{a}hler-Einstein metrics]{An explicit description of the K\"{a}hler-Einstein \\ metrics of Guenancia--Hamenst\"{a}dt}

\author{Jean-Fran\c{c}ois Lafont}
\address{Department of Mathematics, The Ohio State University, Columbus, Ohio 43210}
\email{jlafont@math.ohio-state.edu}

\author{Barry Minemyer}
\address{Department of Mathematics, Computer Science, and Digital Forensics, Commonwealth University - Bloomsburg, Bloomsburg, Pennsylvania 17815}
\email{bminemyer@commonwealthu.edu}

\subjclass[2020]{Primary 53C25, 53C35, 51M15; Secondary 53C55, 53B20, 57R18}

\date{\today.}

\begin{abstract}
Fine and Premoselli (FP) constructed the first examples of manifolds that do not admit a locally symmetric metric but do admit a negatively curved Einstein metric.  
The manifolds here are hyperbolic branched covers like those used by Gromov and Thurston, and the construction of their model Einstein metric is a variation of the hyperbolic metric written in polar coordinates.  
Very recently, Guenancia and Hamenst\"{a}dt (GH) proved the existence of the first examples of manifolds that are not locally symmetric but admit a negatively curved K\"{a}hler-Einstein metric. 
The GH metrics are realized on complex hyperbolic branched covers constructed by Stover and Toledo.
In this article we generalize the construction of FP to the complex hyperbolic setting and show that this yields a negatively curved Einstein metric that asymptotically approaches the metric of GH.
\end{abstract}

\maketitle

\section{Introduction}\label{Section:Introduction}

Prior to 2020, the only known examples of closed manifolds which admit an Einstein metric with negative sectional curvature were compact quotients of real, complex, quaternionic, and Cayley hyperbolic spaces.  
Then, in \cite{FP}, Fine and Premoselli constructed the first examples of compact $4$-manifolds which admit a negatively curved Einstein metric, but which are not homotopy equivalent to a quotient of a locally symmetric space (see also the corresponding conference proceedings by Premoselli \cite{Premoselli}).

A brief outline of the construction in \cite{FP} is as follows.
Fine and Premoselli construct a sequence of compact hyperbolic $4$-manifolds $M_k$, each containing an embedded, totally geodesic, (possibly disconnected) codimension two submanifold $N_k$ which is nullhomologous in $M_k$.
This sequence is arranged so that, as $k \to \infty$, the normal injectivity radius $\eta_k$ of $N_k$ in $M_k$ also approaches infinity.
Fix an integer $d \geq 2$ and let $X_k$ denote the $d$-fold cyclic branched cover of $M_k$ about $N_k$.
These manifolds $X_k$ are specific examples of the hyperbolic branched cover manifolds considered by Gromov and Thurston in \cite{GT}, and so the more interesting portion of their construction is the negatively curved Einstein metric.  
Fine and Premoselli give a very explicit construction of a negatively curved orbifold model Einstein metric defined on the $\eta_k$-neighborhood of $N_k$.
This metric $\l_k$ is everywhere Riemannian except it has a cone angle of $2 \pi / d$ about $N_k$, and so $\l_k$ pulls back to a smooth metric within this same neighborhood of the ramification locus of $X_k$.
On the $((1/2)\eta_k, \eta_k)$-annulus of the ramification locus, Fine and Premoselli use a smooth cut-off function to interpolate between $\l_k$ and the hyperbolic metric $\y_4$ on each page of the cover.  
For $k$ large, this interpolated metric $g_k$ is negatively curved and approximately Einstein.  
Fine and Premoselli then construct an inverse function theorem to show that, for $k$ sufficiently large, this metric can be perturbed to a genuine Einstein metric $\mathfrak{e}_k$ with negative sectional curvature.

In \cite{HJ}, Hamemst\"{a}dt and J\"{a}ckel extend the construction in \cite{FP} to all dimensions.  
Their approach is the same as above, and they use the same model orbifold Einstein metric as in \cite{FP}.
But they use subgroup separability to obtain stronger volume bounds on the submanifolds $N_k$, and they develop an extension of the inverse function theorem mentioned above.  

In \cite{ST}, Stover and Toledo construct complex hyperbolic branched cover manifolds analogous to the hyperbolic branched covers in \cite{GT}.  
Stover and Toledo proved that these manifolds are not homotopy equivalent to a manifold admitting a locally symmetric metric, and by \cite{Zheng} it is known that these manifolds admit a negatively curved K\"{a}hler metric.
Thus, since the first Chern class of this manifold is negative, it admits a unique (up to scaling) K\"{a}hler-Einstein metric with negative Einstein constant by the celebrated work of Aubin \cite{Aubin} and Yau \cite{Yau}.  
Very recently, Guenancia and Hamenst\"{a}dt \cite{GH} proved that this K\"{a}hler-Einstein metric has negative sectional curvature, thus verifying the first known examples of negatively curved K\"{a}hler-Einstein metrics on manifolds that are not locally symmetric.  

The argument in \cite{GH} is similar in nature to that in \cite{FP}, but one major difference is that Guenancia and Hamenst\"{a}dt do not give a direct construction for the model K\"{a}hler-Einstein metric.  
They use \cite{CY} and \cite{YauA} to show that a unique model K\"{a}hler-Einstein metric exists \cite[Theorem 2.2]{GH}, they prove that the coefficient of the horizontal distribution of this metric must satisfy a specific second-order differential equation \cite[Theorem 2.9]{GH}, and then use this fact to give a quantitative description of the sectional curvatures.
They also show that, as the distance from the ramification locus increases, this model K\"{a}hler-Einstein metric converges exponentially to the pullback of the complex hyperbolic metric (\cite[Theorem 2.4]{GH}).
The fact that this model K\"{a}hler-Einstein metric is negatively curved is due to Bland \cite{Bland}.

Prior to the announcement of \cite{GH} the authors were working on this same problem, but taking a slightly different approach.  
Our goal was to take the entire process from \cite{FP} and \cite{HJ} and try to generalize it to the complex hyperbolic setting.  
This has the advantage of giving a very explicit description of the model Einstein metric, but has the downside of making it difficult, if not impossible, to determine whether or not the resulting metric is K\"{a}hler.  

Upon inspection of \cite{GH}, the authors have realized that the natural extension of the model negatively curved Einstein metric from \cite{FP} to the complex hyperbolic setting coincides with the model K\"{a}hler-Einstein metric whose existence is guaranteed by \cite[Thm. 2.2]{GH}.
We state this formally as follows.

\begin{theorem}\label{thm:main theorem}
The negatively curved model Einstein metric from Fine--Premoselli \cite{FP}, generalized to complex hyperbolic branched cover manifolds, produces the model K\"{a}hler-Einstein metric whose existence is guaranteed by Guenancia--Hamenst\"adt \cite[Theorem 2.2]{GH}.
\end{theorem}

Combining our work in Sections \ref{sect:curvature formulas} and \ref{sect: approximate Einstein metric} below with the inverse function theorems of \cite{FP} and \cite{HJ} proves the following.  

\begin{theorem}\label{thm:main theorem 2}
There exist K\"{a}hler manifolds of every complex dimension which admit a negatively curved Einstein metric but do not admit a locally symmetric metric.
\end{theorem}

\begin{remark}
The authors want to make it clear that they are, in no way, trying to take or share any credit for the results in \cite{GH}.
The authors were working on this project before the announcement of \cite{GH}, but our work was not finished.
We have chosen to publish the results that we did have because we feel that they both augment \cite{GH} well, and that they are useful in their own right.
The math contained in Sections \ref{sect:curvature formulas} and \ref{sect: approximate Einstein metric} below was done independently and completed before the announcement of \cite{GH} (besides the comments tying our work to \cite{GH}, of course).  
Also, while the authors had not yet discussed the construction of the manifolds in Subsection \ref{sect:construction of manifolds}.1 below, this seemed reasonably implicit from \cite{ST} and with the idea of using subgroup separability from \cite{HJ}.
The second author briefly outlined a similar construction in the Introduction of \cite{MinemyerKahler}.
But, while the authors believed that they could adapt a version of the inverse function theorem used in \cite{HJ} or \cite{Jackel}, they had not yet looked into this at all.  
And, more importantly, prior to the announcement of \cite{GH} the authors did not know that the construction in this paper may lead to a K\"{a}hler metric.
\end{remark}

The authors believe that the main contribution of this paper to the literature is the actual construction of the model Einstein metric (see equations \eqref{eqn:lambda V} and \eqref{eqn:V equation}).  
The complex hyperbolic branched cover manifolds from \cite{ST} are known to be K\"{a}hler by the work of Zheng \cite{Zheng} and now from \cite{GH}.  
But, from a differential geometry point of view, neither of these papers give a very explicit description of the metric.  
Our work in Sections \ref{sect:curvature formulas} and \ref{sect: approximate Einstein metric} below gives a precise and ``hands on" description of such a metric that one could easily work with in other settings if needed.  

One interesting question though is whether or not the Einstein metric $\mathfrak{e}_k$ whose existence is guaranteed by Theorem \ref{thm:main theorem 2} must be the K\"{a}hler-Einstein metric $\omega_k$ from \cite{GH} for $k$ sufficiently large?
Using equation $(30)$ from \cite{GH} we can show that $\mathfrak{e}_k$ $C^2$-approaches $\omega_k$ (see Remark \ref{rmk:einstein approximation} below).  
But it is unknown if $\omega_k$ is isolated in which case $\mathfrak{e}_k$ must eventually equal $\omega_k$, or if there exists a dense collection of negatively curved Einstein metrics about $\omega_k$?

One potential application of the results of this paper is the following.  
Assume that the dimension of $X_k$ is divisible by $4$.  
Via Chern-Weil theory, the explicit curvature formulas from Thereom \ref{thm:curvature formulas} could be used to compute the signature $\Sigma(X_k)$ of $X_k$.  
Hirzebruch \cite{Hirzebruch-RamifiedCoverings} and Viro \cite{Viro} developed formulas that relate $\Sigma(X_k)$ to the Euler number of the normal bundle of the ramification locus.  
In this way, one may be able to obtain information about (at least) the top Chern class of the Stover-Toledo manifolds $X_k$.  
At present there seems to be nothing known about the invariants of $X_k$.

\vskip 10pt

The remainder of this paper is laid out as follows.
In Section \ref{sect:curvature formulas} we review polar coordinates with respect to the complex hyperbolic metric, we consider the warped-product metric $\l$ that corresponds to \cite{FP}, and we calculate formulas for the Riemann curvature tensor of $\l$.
Due to the existence of nonzero mixed terms, this curvature calculation is considerably more involved than what is required in \cite{FP}.  
Section \ref{sect: approximate Einstein metric} mirrors Section 3 of \cite{FP}.  
We calculate the Ricci tensor of $\l$, determine values for the warping function $V_\k(u)$ for which the corresponding $\l_\k$ will be Einstein, we consider the cone angle for different choices of $\k$, and we show that $\l_\k$ is negatively curved with respect to these choices of $V_\k$.  
We also prove that $\l_\k$ is the same metric as $\omega_\a$ from \cite{GH}.
In Section \ref{sect:construction of manifolds} we briefly outline the construction of the sequence of complex hyperbolic branched cover manifolds.  
We then explain how to taper $\l_\k$ to the complex hyperbolic metric $\cm_n$ on each page of the branched cover to obtain an approximate K\"{a}hler-Einsten metric, and we discuss how this construction relates to the work in \cite{FP} and \cite{GH}.  

One final comment is that Section 2 of \cite{GH} is devoted to showing that the model K\"{a}hler-Einstein metric $\omega_\a$ is negatively curved and approaches the pullback of the complex hyperbolic metric exponentially.  
In Proposition \ref{prop:lambda negative curvature} we give a direct proof that our metric $\l_\k$ is negatively curved, and in the ensuing remark we point out how $\l_\k$ exponentially approaches $\cm_n$.  
In light of the fact that $\l_\k$ is the same metric as $\omega_\a$, this gives an alternate proof of \cite[Theorem 2.11]{GH}.
This also extends some of the results of Bland in \cite{Bland}.

\subsection*{Acknowledgments} The authors would like to thank Matthew Stover and Domingo Toledo for sharing their interest in these manifolds. We would also like to thank Ursula Hamenst\"adt for graciously encouraging us to complete this manuscript. The  authors were partially supported by the NSF, under grant DMS-2407438.

\section{The Metric $\l$ and Curvature Formulas}\label{sect:curvature formulas}
The purpose of this section is to develop curvature formulas (Theorem \ref{thm:curvature formulas} below) which we will use to build our negatively curved model Einstein metric. 
In this section, all calculations are completed in the universal cover $\C \H^n$.

\subsection{The complex hyperbolic metric written in polar coordinates}\label{subsection:chn/chn-1}
All of the results in this subsection were discovered by Belegradek in \cite{belegradek2012complex} and rescaled to have constant holomorphic curvature of $-4$ in \cite{Minemyercurvatureformulas}.
Let $\C \H^n$ denote complex hyperbolic space of dimension $n$ normalized to have constant holomorphic sectional curvature $-4$, and let $q \in \C \H^n$.  
With respect to an appropriately chosen frame about $q$, the standard complex hyperbolic metric $\cm_n$ can be written as
\begin{align}
	\cm_n &= \cosh^2(r) \cm_{n-1} + \frac{1}{4} \sinh^2(2r) d \th^2 + dr^2 \notag \\
	&= \cosh^2(r) \cm_{n-1} + \cosh^2(r) \sinh^2(r) d \th^2 + dr^2 \label{eqn:c metric}
	\end{align}
We now describe the implicit frame for this representation of $\cm_n$.

Let $\C \H^{n-1}$ denote a totally geodesic copy of complex hyperbolic $(n-1)$-space in $\C \H^n$, and let $\phi : \C \H^n \to \C \H^{n-1}$ denote the orthogonal projection map.  
Let $p = \phi(q)$, and let $r$ denote the distance from a given point to $\C \H^{n-1}$.  
The space $\phi^{-1}(p)$ is a totally geodesic copy of $\C \H^1$ which contains $q$.  
There is a $2$-plane $\s \subset T_q \C \H^n$ such that $\text{exp}_q(\s) = \phi^{-1}(p)$.  
Using polar coordinates $(\partial / \partial \th , \partial / \partial r)$ on $\s$, we have
\begin{equation*}
\cm_n \bigr|_\s = \frac{1}{4} \sinh^2(2r) d \th^2 + dr^2.
\end{equation*}
Now, let $\mathcal{H} = \sigma^{\perp} \subset T_q \C \H^n$ denote the {\it horizontal distribution} of $\cm_n$ at $q$.  
The map
\begin{equation*}
	\varphi := d \phi_q \bigr|_{\mathcal{H}}: \mathcal{H} \to T_p \C \H^{n-1}
\end{equation*}
is an isomorphism.
If we let $(\check{X}_i)_{i=1}^{2n-2}$ be an orthonormal basis for $T_p \C \H^{n-1}$, then $( \varphi^{-1}(\check{X}_i ))_{i=1}^{2n-2}$ is an orthogonal basis for $\mathcal{H}$.
Let $X_i = \varphi^{-1}(\check{X}_i)$ for each $i$.  
The basis $(X_i)$ is not orthonormal:  we have $\cm_n (X_i, X_i) = \cosh^2(r)$ for each $i$.  
Putting this all together gives the desired frame for \eqref{eqn:c metric} and, moreover, gives a natural diffeomorphism $\C \H^n \setminus \C \H^{n-1} \cong \R^{2n-2} \times \S^1 \times (0, \infty)$.  

A very important feature of this decomposition of $\C \H^n$ is that the horizontal distribution is not integrable with respect to $\cm_n$.  
Consider two orthonormal vectors $\check{X}_i, \check{X}_j \in T_p \C \H^{n-1}$, and extend them to a frame near $p$ (in $\C \H^{n-1}$) in such a way that $[\check{X}_i, \check{X}_j]_p = 0$.  
As above, let $X_i = \varphi^{-1}(\check{X}_i)$ and $X_j = \varphi^{-1}(\check{X}_j)$.
It was proved in \cite{belegradek2012complex} for constant holomorphic curvature $-1$ and rescaled in \cite{Minemyercurvatureformulas} for curvature $-4$ that $[X_i, X_j]_q = 0$ if and only if $(\check{X}_i, \check{X}_j)$ spans a totally real $2$-plane in $T_p \C \H^{n-1}$, whereas $[X_i, X_j]_q = \pm 2 (\partial / \partial \th )$ if $(\check{X}_i, \check{X}_j)$ spans a complex line.  

In our curvature calculations below it will be convenient to use a {\it holomorphic frame} $(X_i)_{i=1}^{2n-2}$ for $\mathcal{H}$, which we define as follows.  
Let $\check{X}_1 \in T_p \C \H^{n-1}$ be any unit vector and define $\check{X}_2 = J \check{X}_1$, where $J$ denotes the complex structure on $\C \H^n$.  
We will refer to such a pair $(\check{X}_i, \check{X}_j)$ with $J \check{X}_i = \pm \check{X}_j$ as a {\it holomorphic pair}.  
Then, assuming that $\check{X}_1, \hdots, \check{X}_{2k}$ have been defined, we define $\check{X}_{2k+1}$ to be any unit vector orthogonal to $\text{span}(\check{X}_1, \hdots, \check{X}_{2k})$ and $\check{X}_{2k+2} = J \check{X}_{2k+1}$.  
It is an easy exercise to check that $\check{X}_{2k+2}$ is in the orthogonal complement of $\text{span}(\check{X}_1, \hdots, \check{X}_{2k}, \check{X}_{2k+1})$.  
We then extend this basis to a frame near $p$ in $\C \H^{n-1}$ in such a way that $[\check{X}_i, \check{X}_j]_p = 0$ for all $i, j$, and define $X_i = \varphi^{-1}(\check{X}_i)$ for each $i$.  

The non-integrability of $\mathcal{H}$ can be described in terms of a holomorphic frame as follows.  
For a holomorphic frame $(X_i)_{i=1}^{2n-2}$ about $q$ we have
\begin{equation}\label{eqn:nonzero Lie brackets}
	[X_i, X_{i+1}]_q = 2 \ddt \qquad \text{for } 1 \leq i \leq 2n-2 \text{ and } i \text{ odd}
\end{equation}
and, if $(X_i, X_j)$ is not a holomorphic pair, then
\begin{equation*}
	[X_i, X_j]_q = 0 .
\end{equation*}

\subsection{Coordinate Change and Curvature Formulas}
Following \cite{FP}, we make the substitution
\begin{equation*}
	u = \cosh(r)
\end{equation*}
into equation \eqref{eqn:c metric}.
This yields
\begin{equation}\label{eqn:u coord change metric}
\cm_n = u^2 \cm_{n-1} + u^2 (u^2 - 1) d \th^2 + \frac{1}{u^2 - 1} du^2.
\end{equation}
We wish to consider the variable metric $\l$ obtained from \eqref{eqn:u coord change metric} by replacing $u^2 - 1$ with a generic smooth, positive function $V(u)$.  
This gives
\begin{equation}\label{eqn:lambda V}
	\l = u^2 \cm_{n-1} + u^2 V d \th^2 + \frac{1}{V} du^2.
\end{equation}
Of course, properties of $\l$ will depend on the chosen function $V$.  
For calculations it is convenient to set $W = \sqrt{V}$, which gives
\begin{equation}\label{eqn:lambda}
\l = u^2 \cm_{n-1} + u^2 W^2 d \th^2 + \frac{1}{W^2} du^2.	
\end{equation}

The metric $\l$ has three differences with the metric $g$ from Section 3.1 of \cite{FP}.  
The two obvious differences are the addition of the $u^2$ with the $d \th^2$ term and that the horizontal distribution is a warped copy of $\cm_{n-1}$ instead of the hyperbolic metric $\y_{n-2}$.  
The third, less obvious, difference is the non-integrability of $\mathcal{H}$ with respect to the complex hyperbolic metric.  
Despite these differences the curvature formulas for $\l$ are, surprisingly, very similar to those of $g$ from \cite{FP}.

Let us set up an orthonormal frame for these curvature formulas.  
Let $(X_i)_{i=1}^{2n-2}$ be a holomorphic frame for $\mathcal{H}$, and let $\partial / \partial \th$ and $\partial / \partial r$ be as defined above.  
Set
\begin{align}
Y_i &= \frac{1}{u} X_i \qquad 1 \leq i \leq 2n-2 \label{eqn:ON basis 1} \\
Y_{2n-1} &= \frac{1}{uW} \ddt \\
Y_{2n} &= W \frac{\partial}{\partial u}.	\label{eqn: ON basis 3}
\end{align}
Lastly, if $R^\l$ denotes the Riemann curvature tensor of $\l$, we set
\begin{equation*}
R^\l_{i,j,k,l} = \l( R^\l(Y_i, Y_j)Y_k, Y_l).
\end{equation*}

\begin{theorem}\label{thm:curvature formulas}
With respect to the orthonormal basis defined in equations \eqref{eqn:ON basis 1} through \eqref{eqn: ON basis 3} we have, up to the symmetries of the curvature tensor, the following formulas for the nonzero components of the Riemann curvature tensor $R^\l$.
\begin{align}
&R^\l_{i,i+1,i,i+1} = -4 \left( \frac{1 + W^2}{u^2} \right) \hskip 2pt \text{ for } i \text{ odd.} \label{eqn:Thm 2.1 eqn 1} \\
&R^\l_{i,j,i,j} = - \frac{1 + W^2}{u^2}. \label{eqn:Thm 2.1 eqn 2}  \\
&R^\l_{i, 2n-1, i, 2n-1} = R^\l_{i,2n,i,2n} = - \frac{W W'}{u}. \label{eqn:Thm 2.1 eqn 3}  \\
&R^\l_{2n-1,2n,2n-1,2n} = -3 \frac{W W'}{u} - \left( W W'' + \left( W' \right)^2 \right). \label{eqn:Thm 2.1 eqn 4}  \\
&R^\l_{i, i+1, j, j+1} = 2 R^\l_{i, j, i+1, j+1} = -2 R^\l_{i, j+1, i+1, j}= -2 \left( \frac{1 + W^2}{u^2} \right) \hskip 2pt \text{ for } i, j \text{ odd}. \label{eqn:Thm 2.1 eqn 5} \\
&R^\l_{i, i+1, 2n-1, 2n} = 2R^\l_{i, 2n-1, i+1, 2n} = -2R^\l_{i, 2n, i+1, 2n-1}= -2 \frac{W W'}{u} \hskip 2pt \text{ for } i \text{ odd}. \label{eqn:Thm 2.1 eqn 6}
\end{align}
In the formulas above we have $1 \leq i, j \leq 2n-2$ and, if $i$ and $j$ appear in the same formula, we assume that $(Y_i, Y_j)$ does not form a holomorphic pair.
\end{theorem}

It is a simple calculation to check that, when $W = \sqrt{u^2 - 1}$, the formulas in Theorem \ref{thm:curvature formulas} reduce to either $-4$, $-1$, or $-2$ respectively (see equations $(4.4)$ through $(4.6)$ of \cite{Minemyercurvatureformulas} for the correct curvature values with respect to $\cm_n$).

The remainder of this section is devoted to proving Theorem \ref{thm:curvature formulas}.  
The method used in proving Proposition 3.2 of \cite{FP} seems significantly more complicated due to the non-integrability of $E$, which is expressed by the Lie brackets in \eqref{eqn:nonzero Lie brackets} and contributes to the nonzero mixed terms of $R^\l$ listed above.
Additionally, the formulas developed by Belegradek in \cite{belegradek2012complex} do not seem to work due to the $1/W^2$ term in front of the $du^2$.  
Thus, we calculate the formulas for $R^\l$ directly.  
How we proceed is that we first explain why we can restrict to $n=3$.  
Then, using the results of \cite{minemyer2018real}, we calculate formulas for the Lie brackets for our basis.  
Finally, using these Lie bracket formulas we calculate formulas for the Levi-Civita connection, and then use the connection to calculate the above formulas for the Riemann curvature tensor.  
The reader who is uninterested in these calculations and only cares about the use of Theorem \ref{thm:curvature formulas} can skip to Section \ref{sect: approximate Einstein metric}.

\subsection{Reduction to $n=3$}
Curvature formulas for $\cm_n$ in equation \eqref{eqn:c metric} with respect to a holomorphic frame are worked out in \cite{belegradek2012complex} and \cite{Minemyercurvatureformulas}.  
In making the substitution $u =\cosh(r)$ to obtain \eqref{eqn:u coord change metric}, the only change in our implicit frame is that $\partial / \partial u$ is now a (positive) scaled copy of $\partial / \partial r$. 
So, in particular, equations \eqref{eqn:ON basis 1} through \eqref{eqn: ON basis 3} still provide a holomorphic frame for \eqref{eqn:u coord change metric}.  

A {\it mixed term} of a Riemann curvature tensor, with respect to a given fixed orthonormal frame, is any term that does not correspond to the sectional curvature of a coordinate plane with respect to the given frame.  
Equivalently, this is any term that is not of the form $R_{i, j, i, j}$ for some $i$ and $j$.
With respect to a holomorphic frame, a mixed term $R^{\cm_n}_{a,b,c,d}$ of the curvature tensor $R^{\cm_n}$ of $\cm_n$ is nonzero only when the respective vectors $\{ Y_a, Y_b, Y_c, Y_d \}$ contain exactly two holomorphic pairs.  
In other words, if $\{ Y_a, Y_b, Y_c, Y_d \}$ is not of the form $\{ Y_i, Y_{i+1}, Y_j, Y_{j+1} \}$ for some $i, j$ odd, then the mixed term $R^{\cm_n}_{a,b,c,d} = 0$.
This property is then inherited by the warped-product metric $\l$ from \eqref{eqn:lambda}.  

There are two distinct ways in which the collection $\{ Y_i, Y_{i+1}, Y_j, Y_{j+1} \}$ can contain two holomorphic pairs:
\begin{enumerate}
	\item Both $1 \leq i, j \leq 2n-2$.  In this case, both pairs are contained in the horizontal distribution $\mathcal{H}$.  
	\item $i = 2n-1$ or $j = 2n-1$.  In this case, one pair comes from $\mathcal{H}$ and the other pair has the same span as $\{ \partial / \partial \th , \partial / \partial u \}$.  
\end{enumerate}
All of these cases are covered when the complex dimension of $\mathcal{H}$ is at least two or, equivalently, when $n \geq 3$.  
Therefore, to simplify calculations below, we will restrict to $n=3$ which, again, covers all possible nonzero values for $R^\l$.
Note that, when $n=3$, the horizontal distribution $\mathcal{H}$ is a scaled copy of $\C \H^2$.

\subsection{A frame and Lie Brackets for $\C \H^2$}
Direct calculations of curvature formulas in complex hyperbolic space can be very difficult.  
In \cite{minemyer2018real} the second author gave a direct calculation to obtain curvature formulas for $\C \H^2$ with respect to polar coordinates about a totally geodesic (and totally real)  copy of the hyperbolic plane $\H^2$.  
Using \cite{minemyer2018real} seems to be the easiest way to calculate the curvatures in Theorem \ref{thm:curvature formulas}.
In this subsection we give a quick overview of the necessary results from \cite{minemyer2018real}.
Note that in \cite{minemyer2018real} curvatures were scaled to lie in $[-1, -1/4]$.  
These formulas were scaled for curvatures in $[-4,-1]$ in \cite{Minemyercurvatureformulas}.

To avoid confusion\footnote{The variables $(\th, r)$ were used in \cite{minemyer2018real}, but these variables were already used in \eqref{eqn:c metric} above.}, we will use $(\tau, \s)$ to denote polar coordinates in $\C \H^2$ about $\H^2$.  
The horizontal distribution in this setting is spanned by two vectors denoted $S$ and $T$, where $S$ is parallel to $ J (\partial / \partial \tau)$ and $T$ is parallel to $J(\partial / \partial \s)$.  
The complex hyperbolic metric $\cm_2$ can then be written as
\begin{equation*}
\cm_2 = \cosh^2(\s) dS^2 + \sinh^2(\s) d \tau^2 + \cosh^2(2\s) dT^2 + d \sigma^2.
\end{equation*}
Let
\begin{equation}\label{eqn:ch2 X check frame}
\check{X}_1 = S \qquad \check{X}_2 = \frac{\partial}{\partial \t} \qquad \check{X}_3 = T \qquad \check{X}_4 = \frac{\partial}{\partial \s}.	
\end{equation}
It should be noted that this swaps the definitions of $X_1$ and $X_2$ from \cite{minemyer2018real}.  
The nonzero Lie brackets with respect to this frame are
\begin{equation}\label{eqn:ch2 X check Lie brackets}
[\check{X}_1, \check{X}_2]_{\cm_2} = \check{X}_3 \qquad [\check{X}_1, \check{X}_3]_{\cm_2} = \check{X}_2 \qquad [\check{X}_2, \check{X}_3]_{\cm_2} = \check{X}_1.	
\end{equation}
Notice that we have included a $\cm_2$ subscript to distinguish these Lie brackets from the Lie brackets for $\l$ in the following subsection.
We now define our orthonormal frame for $\cm_2$ by
\begin{equation}\label{eqn:ch2 X frame}
X_1 = \frac{1}{\cosh(\s)} \check{X}_1, \quad X_2 = \frac{1}{\sinh(\s)} \check{X}_2, \quad X_3 = \frac{1}{\cosh(2 \s)} \check{X}_3, \quad X_4 = \check{X}_4.
\end{equation}
A direct calculation, where one remembers that $X_4 = \partial / \partial \s$, gives
\begin{align}
&[X_1, X_2]_{\cm_2} = \frac{\cosh(2 \s)}{\sinh(\s) \cosh(\s)} X_3 \qquad [X_1, X_3]_{\cm_2} = \frac{\sinh(\s)}{\cosh(\s)\cosh(2\s)} X_2  \label{eqn:ch2 X Lie brackets 1}\\
&[X_1, X_4]_{\cm_2} = \frac{\sinh(\s)}{\cosh(\s)} X_1 \qquad \hskip 34pt [X_2, X_3]_{\cm_2} = \frac{\cosh(\s)}{\sinh(\s)\cosh(2\s)} X_1 \\
&[X_2, X_4]_{\cm_2} = \frac{\cosh(\s)}{\sinh(\s)} X_2 \qquad \hskip 34pt [X_3, X_4]_{\cm_2} = \frac{2 \sinh(2\s)}{\cosh(2 \s)} X_3.  \label{eqn:ch2 X Lie brackets 3}
\end{align}

\subsection{Lie brackets for the frame \eqref{eqn:ON basis 1} through \eqref{eqn: ON basis 3}}
We now return our attention to the metric $\l$ from \eqref{eqn:lambda} with the simplifying assumption that $n=3$.  
We will use the basis $(X_1, X_2, X_3, X_4)$ from \eqref{eqn:ch2 X frame} as our orthogonal frame for the horizontal distribution $\mathcal{H}$.
We extend this to an orthogonal basis for $\l$ by setting
\begin{equation*}
X_5 = \ddt \qquad X_6 = \frac{\partial}{\partial u}.
\end{equation*}
Recall from equation \eqref{eqn:nonzero Lie brackets} that, since $(X_1, X_2)$ and $(X_3, X_4)$ form holomorphic pairs, both $[X_1, X_2]$ and $[X_3, X_4]$ have an $X_5$ component of $2 X_5$.

Define the corresponding orthonormal basis $(Y_i)$ for $\l$ by
\begin{equation}\label{eqn:lambda ON basis}
Y_i = \frac{1}{u} X_i \quad (1 \leq i \leq 4) \qquad Y_5 = \frac{1}{uW} X_5 \qquad Y_6 = W X_6.
\end{equation}
A direct calculation, using equations \eqref{eqn:ch2 X Lie brackets 1} through \eqref{eqn:ch2 X Lie brackets 3} and remembering that $Y_6 = W (\partial / \partial u)$, yields the following Lie brackets.  
For ease of notation we make the substitutions $a = \cosh(\s)$, $b = \sinh(\s)$, and $c = \cosh(2 \s)$.  

\begin{proposition}\label{prop:lambda Lie brackets}
The values for the Lie brackets of $\l$ with respect to the frame given in \eqref{eqn:lambda ON basis} are
\begin{align*}
&[Y_1, Y_2] = \frac{c}{uab} Y_3 + \frac{2W}{u} Y_5 \qquad [Y_2, Y_3] = \frac{a}{ubc} Y_1 \qquad \hskip 40pt [Y_3, Y_5] = 0  \\
&[Y_1, Y_3] = \frac{b}{uac} Y_2 \qquad \hskip 40pt [Y_2, Y_4] = \frac{a}{ub}Y_2 \qquad \hskip 46pt [Y_3, Y_6] = \frac{W}{u} Y_3  \\
&[Y_1, Y_4] = \frac{b}{ua} Y_1 \qquad \hskip 44pt [Y_2, Y_5] = 0 \qquad \hskip 64pt [Y_4, Y_5] = 0  \\
&[Y_1, Y_5] = 0 \qquad \hskip 62pt [Y_2, Y_6] = \frac{W}{u} Y_2 \qquad \hskip 48pt [Y_4, Y_6] = \frac{W}{u} Y_4  \\
&[Y_1, Y_6] = \frac{W}{u} Y_1 \qquad \hskip 44pt [Y_3, Y_4] = \frac{4ab}{uc} Y_3 + \frac{2W}{u} Y_5 \qquad [Y_5, Y_6] = \frac{W + u W'}{u} Y_5
\end{align*}
\end{proposition}

\begin{proof}
To demonstrate the techniques used to calculate these Lie brackets, we calculate $[Y_1, Y_2]$ and $[Y_5, Y_6]$ and leave the remainder to the reader.  
For this first Lie bracket, we have
\begin{equation*}
[Y_1, Y_2] = \frac{1}{u^2} [X_1, X_2] = \frac{1}{u^2} \left( \frac{c}{ab} X_3 + 2 X_5 \right) = \frac{c}{uab} Y_3 + \frac{2W}{u} Y_5.
\end{equation*}
For the second Lie bracket, we have to remember to calculate the partial derivative since $Y_6 = W (\partial / \partial u)$.  
Thus, we obtain
\begin{align*}
[Y_5, Y_6] &= W \left[ \frac{1}{uW} X_5 , \frac{\partial}{\partial u} \right] = - W \frac{\partial}{\partial u} \left( \frac{1}{uW} \right) X_5 = - W \left( - \frac{W + uW'}{u^2 W^2} \right) X_5 \\
&= \frac{W + u W'}{u} Y_5 .
\end{align*}
\end{proof}

\subsection{The Levi-Civita connection and the proof of Theorem \ref{thm:curvature formulas}}
Recall that, in the special case of an orthonormal basis, the Koszul formula reduces to 
\begin{equation*}
\l(\nabla_Y X, Z) = - \frac{1}{2} \biggl( \l \big([X,Z],Y\big) + \l \big([Y,Z],X\big) + \l \big([X,Y], Z \big) \biggr).
\end{equation*}
Combining this equation with the formulas in Proposition \ref{prop:lambda Lie brackets} provides the following $36$ components of the Levi-Civita connection.

\begin{proposition}\label{prop:connection 1}
The Levi-Civita connection $\nabla$ compatible with $\l$ is determined by the following $36$ equations.
	\begin{align*}
	\bullet \nabla_{Y_1}Y_1 = & - \frac{b}{ua} Y_4 - \frac{W}{u} Y_6 \hskip 56pt \bullet \nabla_{Y_1}Y_2 =  - \frac{1}{2u}\left( \frac{a}{bc} + \frac{b}{ac} - \frac{c}{ab} \right) Y_3 + \frac{W}{u} Y_5    \\
	\bullet \nabla_{Y_1}Y_3 = & - \frac{1}{2u} \left( - \frac{a}{bc} - \frac{b}{ac} + \frac{c}{ab} \right) Y_2  \hskip 124pt	  \bullet \nabla_{Y_1}Y_4 = \frac{b}{ua} Y_1  \\
	\bullet \nabla_{Y_2}Y_1 = & - \frac{1}{2u} \left( \frac{a}{bc} + \frac{b}{ac} + \frac{c}{ab} \right) Y_3 - \frac{W}{u} Y_5    \hskip 54pt  \bullet \nabla_{Y_2}Y_2 = - \frac{a}{ub} Y_4 - \frac{W}{u} Y_6   \\
	\bullet \nabla_{Y_2}Y_3 = & - \frac{1}{2u} \left( - \frac{a}{bc} - \frac{b}{ac} - \frac{c}{ab} \right) Y_1  \hskip 126pt  \bullet \nabla_{Y_2}Y_4 = \frac{a}{ub} Y_2   \\ 
	\bullet \nabla_{Y_3}Y_1 = & - \frac{1}{2u} \left( - \frac{a}{bc} + \frac{b}{ac} + \frac{c}{ab} \right) Y_2	\hskip 40pt  \bullet \nabla_{Y_3}Y_2 = - \frac{1}{2u} \left( \frac{a}{bc} - \frac{b}{ac} - \frac{c}{ab} \right) Y_1    \\
	\bullet \nabla_{Y_3}Y_3 = & - \frac{4ab}{uc}Y_4 - \frac{W}{u} Y_6  \hskip 132pt  \bullet \nabla_{Y_3}Y_4 = \frac{4ab}{uc} Y_3 + \frac{W}{u} Y_5    \\
	\bullet \nabla_{Y_4}Y_1 = & 0   \hskip 32pt  \bullet \nabla_{Y_4}Y_2 = 0 \hskip 32pt \bullet \nabla_{Y_4}Y_3 = - \frac{W}{u} Y_5  \hskip 32pt  \bullet \nabla_{Y_4}Y_4 = - \frac{W}{u} Y_6  \\
	\bullet \nabla_{Y_1}Y_5 = & - \frac{W}{u} Y_2  \hskip 18pt  \bullet \nabla_{Y_2}Y_5 = \frac{W}{u} Y_1 \hskip 18pt  \bullet \nabla_{Y_3}Y_5 = - \frac{W}{u} Y_4 \hskip 18pt  \bullet \nabla_{Y_4}Y_5 = \frac{W}{u} Y_3  \\
	\bullet \nabla_{Y_1}Y_6 = &\frac{W}{u} Y_1	\hskip 25pt	\bullet \nabla_{Y_2}Y_6 = \frac{W}{u} Y_2	\hskip 25pt	\bullet \nabla_{Y_3}Y_6 = \frac{W}{u}Y_3	\hskip 25pt	\bullet \nabla_{Y_4}Y_6 = \frac{W}{u} Y_4  \\
	\bullet \nabla_{Y_5}Y_1 = & - \frac{W}{u} Y_2 	\hskip 18pt  \bullet \nabla_{Y_5}Y_2 = \frac{W}{u} Y_1  \hskip 18pt \bullet \nabla_{Y_5}Y_3 = - \frac{W}{u} Y_4  \hskip 18pt \bullet \nabla_{Y_5}Y_4 = \frac{W}{u} Y_3  \\
	\bullet \nabla_{Y_5}Y_5 = & - \left( \frac{W + uW'}{u} \right) Y_6 \hskip 106pt 	\bullet \nabla_{Y_5}Y_6 = \frac{W + uW'}{u} Y_5  \\ 
	\bullet \, &0 = \nabla_{Y_6}Y_1 = \nabla_{Y_6}Y_2 = \nabla_{Y_6}Y_3 = \nabla_{Y_6}Y_4 = \nabla_{Y_6}Y_5 = \nabla_{Y_6}Y_6 
	\end{align*}
\end{proposition}

\vskip 10pt

Note that the coefficients of the form
\begin{equation*}
\left( \pm \frac{a}{bc} \pm \frac{b}{ac} \pm \frac{c}{ab} \right)
\end{equation*}
that occur in some terms of the connection are expected as they are the same terms that occur (up to scale) in the curvature formulas in \cite{minemyer2018real}.  

\vskip 5pt

With the formulas for the connection, we can now prove Theorem \ref{thm:curvature formulas}.

\begin{proof}[Proof of Theorem \ref{thm:curvature formulas}]
Theorem \ref{thm:curvature formulas} contains six different curvature formulas.  
We will compute one example for each of these six formulas.  
The interested reader can confirm that the remainder of the formulas work in an identical manner, and that all components which are not listed are $0$.  
In what follows we use the notation $\nabla_i := \nabla_{Y_i}$.

To verify \eqref{eqn:Thm 2.1 eqn 1} we will compute the more tedious $R_{1,2,1,2}$ (as opposed to $R_{3,4,3,4}$) and choose the simpler calculations for the other formulas.  
We have
\begin{align*}
&R_{1,2,1,2} = \l \left( \nabla_2 \nabla_1 Y_1 - \nabla_1 \nabla_2 Y_1 + \nabla_{[Y_1, Y_2]} Y_1, Y_2 \right) \\
&= \l \left( \nabla_2 \left( \frac{-b}{ua} Y_4 - \frac{W}{u} Y_6 \right) - \nabla_1 \left( - \frac{1}{2u} \left( \frac{a}{bc} + \frac{b}{ac} + \frac{c}{ab} \right) Y_3 - \frac{W}{u} Y_5 \right) \right. \\
& \hskip 40pt \left. +\frac{c}{uab} \nabla_3 Y_1 + \frac{2W}{u} \nabla_5 Y_1, Y_2 \right) \\
&= \frac{-1}{u^2} - \frac{W^2}{u^2} - \frac{1}{4u^2} \left(\frac{a}{bc} + \frac{b}{ac} + \frac{c}{ab} \right) \left(\frac{-a}{bc} - \frac{b}{ac} + \frac{c}{ab} \right) - \frac{W^2}{u^2} \\
& \hskip 40pt - \frac{c}{2u^2 ab} \left( \frac{-a}{bc} + \frac{b}{ac} + \frac{c}{ab} \right) - \frac{2W^2}{u^2} \\
&= - \frac{1}{u^2} - \frac{4W^2}{u^2} - \frac{1}{4u^2} \left( \frac{-a^4 - b^4 + 3c^4 - 2 a^2 b^2 - 2 a^2 c^2 + 2 b^2 c^2}{a^2 b^2 c^2} \right).
\end{align*}
Substituting $a = \cosh(\s)$, $b = \sinh(\s)$, and $c = \cosh(2 \s)$, it is a tedious\footnote{Making repeated use of the identities $a^2 + b^2 = c$, $b^2 - a^2 = -1$, and $a^2 b^2 = \frac{1}{4} \sinh^2 (2 \s)$.} exercise in hyperbolic trigonometric identities to show that
\begin{equation*}
\frac{-a^4 - b^4 + 3c^4 - 2 a^2 b^2 - 2 a^2 c^2 + 2 b^2 c^2}{a^2 b^2 c^2} = 12.
\end{equation*}
This gives $\ds{R_{1,2,1,2} = -4 \left( \frac{1 + W^2}{u^2} \right)}$.

To prove \eqref{eqn:Thm 2.1 eqn 2} we calculate $R_{2,4,2,4}$, remembering that $\ds{Y_4 = \frac{1}{u} \frac{\partial}{\partial \sigma}}$:
\begin{align*}
&R_{2,4,2,4} = \l \left( \nabla_4 \nabla_2 Y_2 - \nabla_2 \nabla_4 Y_2 + \nabla_{[Y_2, Y_4]} Y_2, Y_4 \right) \\
&= \l \left( \nabla_4 \left( \frac{-a}{ub} Y_4 - \frac{W}{u} Y_6 \right) - 0 + \frac{a}{ub} \nabla_2 Y_2, Y_4 \right) \\
&= \l \left( \frac{1}{u^2 \sinh^2(\s)} Y_4 - \frac{a}{ub} \cdot \frac{-W}{u} Y_6 - \frac{W^2}{u^2} Y_4 - \frac{\cosh^2(\s)}{u^2 \sinh^2(\s)} Y_4 + \frac{a}{ub} \cdot \frac{-W}{u}Y_6, Y_4 \right) \\
&= - \frac{1 + W^2}{u^2}.
\end{align*}

For equation \eqref{eqn:Thm 2.1 eqn 3} we calculate $R_{4,5,4,5}$:
\begin{align*}
&R_{4,5,4,5} = \l \left( \nabla_5 \nabla_4 Y_4 - \nabla_4 \nabla_5 Y_4 + \nabla_{[Y_4, Y_5]} Y_4, Y_5 \right) \\
&= \l \left( \nabla_5 \left( - \frac{W}{u} Y_6 \right) - \nabla_4 \left( \frac{W}{u} Y_3 \right) + 0, Y_5 \right) \\
&= - \frac{W}{u} \left( \frac{W + uW'}{u} \right)- \frac{W}{u} \left( - \frac{W}{u} \right) = - \frac{W W'}{u}.
\end{align*}

Our only option to prove \eqref{eqn:Thm 2.1 eqn 4} is to calculate $R_{5,6,5,6}$.  
To this end, we have
\begin{align*}
&R_{5,6,5,6} = \l \left( \nabla_6 \nabla_5 Y_5 - \nabla_5 \nabla_6 Y_5 + \nabla_{[Y_5, Y_6]} Y_5, Y_6 \right) \\
&= \l \left( \nabla_6 \left( - \frac{W + uW'}{u} Y_6 \right) - 0 + \frac{W + uW'}{u} \nabla_5 Y_5, Y_6 \right) \\
&= -W \left( \frac{u (W' + W' + uW'') - (W + uW')}{u^2} \right) - \left( \frac{W + uW'}{u} \right)^2 \\
&= -3 \frac{W W'}{u} - \left( W W'' + \left( W' \right)^2 \right).
\end{align*}

For the first mixed term \eqref{eqn:Thm 2.1 eqn 5}, we calculate
\begin{align*}
&R_{1,2,3,4} = \l \left( \nabla_2 \nabla_1 Y_3 - \nabla_1 \nabla_2 Y_3 + \nabla_{[Y_1, Y_2]} Y_3, Y_4 \right) \\
&= \l \left( \nabla_2 \left( \frac{-1}{2u} \left( \frac{-a}{bc} - \frac{b}{ac} + \frac{c}{ab} \right) Y_2 \right) - \nabla_1 \left( \frac{-1}{2u} \left( - \frac{a}{bc} - \frac{b}{ac} - \frac{c}{ab} \right) Y_1 \right) \right. \\
& \hskip 40pt \left. + \frac{c}{uab} \nabla_3 Y_3 + \frac{2W}{u} \nabla_5 Y_3, Y_4 \right) \\
&= \l \left( \frac{-1}{2u} \left( \frac{-a}{bc} - \frac{b}{ac} + \frac{c}{ab} \right) \left( \frac{-a}{ub} Y_4 - \frac{W}{u} Y_6 \right) + \frac{1}{2u} \left( - \frac{a}{bc} - \frac{b}{ac} - \frac{c}{ab} \right) \cdot \right. \\
& \hskip 40pt \left. \left(\frac{-b}{ua} Y_4 - \frac{W}{u} Y_6 \right) + \frac{c}{uab} \left( \frac{-4ab}{uc} Y_4 - \frac{W}{u} Y_6 \right) - \frac{2W^2}{u^2} Y_4, Y_4 \right).
\end{align*}
To calculate $R_{1,2,3,4}$, the $Y_6$ terms are irrelevant since $Y_6$ is orthogonal to $Y_4$.  
But notice that, by a direct algebra calculation, the coefficients of the $Y_6$ term sum to $0$ (this is actually needed in the computation $R_{1,2,3,6} = 0$).
Returning our attention to the current calculation, the coefficients of the $Y_4$ term simplify to
\begin{align*}
\frac{1}{2u^2} \left( - \frac{a^2}{b^2 c} + \frac{ac}{ab^2} + \frac{bc}{a^2 b} + \frac{b^2}{a^2 c} \right) - \frac{4}{u^2} - 2 \frac{W^2}{u^2}.
\end{align*}
Another exercise in hyperbolic trig identities shows that
\begin{equation*}
- \frac{a^2}{b^2 c} + \frac{ac}{ab^2} + \frac{bc}{a^2 b} + \frac{b^2}{a^2 c} = 4
\end{equation*}
which verifies equation \eqref{eqn:Thm 2.1 eqn 5}.

Finally, to prove \eqref{eqn:Thm 2.1 eqn 6} we calculate
\begin{align*}
&R_{3,4,5,6} = \l \left( \nabla_4 \nabla_3 Y_5 - \nabla_3 \nabla_4 Y_5 + \nabla_{[Y_3, Y_4]} Y_5, Y_6 \right) \\
&= \l \left( \nabla_4 \left( - \frac{W}{u} Y_4 \right) - \nabla_3 \left( \frac{W}{u} Y_3 \right) + \frac{4ab}{uc} \nabla_3 Y_5 + \frac{2W}{u} \nabla_5 Y_5, Y_6 \right) \\
&= \l \left( \frac{W^2}{u^2} Y_6 - \frac{W}{u} \left( \frac{-4ab}{uc} Y_4 - \frac{W}{u} Y_6 \right) + \frac{4ab}{uc} \cdot \frac{-W}{u} Y_4 \right. \\
& \hskip 40pt \left. + \frac{2W}{u} \left( - \frac{W + uW'}{u} \right)Y_6, Y_6 \right) \\
&= -2 \frac{W W'}{u}.
\end{align*}

\end{proof}

\section{The Model Einstein Metric}\label{sect: approximate Einstein metric}
This section very closely mirrors Section 3 of \cite{FP}. 
In this section we determine the collection of all functions $V$ for which $\l$ is Einstein with constant $-(2n+2)$.  
This collection of functions $V$ is parameterized by a choice of constant $\k$.  
We analyze the various cone angles that are created for different choices of $\k$, we prove that $\l_\k$ is equal to $\omega_\a$ from \cite{GH}, and we give an independent proof that $\l_\k$ is negatively curved. 

\subsection{Ricci Tensor and the family of model Einstein metrics}
First recall that, given a point $q \in M$ and an orthonormal basis $(Y_i)_{i=1}^{2n}$ of $T_qM$, the Ricci tensor with respect to the metric $\l$ is the function $\text{Ric}_q: T_q M \times T_qM \to \R$ given by
\begin{equation*}
\text{Ric}_q (A,B) = \sum_{i=1}^{2n} \l_q \left( R(A, Y_i) B, Y_i \right).
\end{equation*}
When the point $q$ is understood from context, we will omit it from the notation.

A direct calculation proves the following.  

\begin{proposition}
With respect to the orthonormal basis $(Y_i)_{i=1}^{2n}$ defined in equations \eqref{eqn:ON basis 1} through \eqref{eqn: ON basis 3}, the values of the Ricci tensor $\text{Ric}$ for $\l$ are
\begin{align*}
&\text{Ric}(Y_i, Y_i) = -2n \left( \frac{1+W^2}{u^2} \right) - 2 \frac{W W'}{u} \quad \text{ for } 1 \leq i \leq 2n-2  \\
&\text{Ric}(Y_{2n-1}, Y_{2n-1}) = \text{Ric}(Y_{2n}, Y_{2n}) = - (2n+1) \frac{W W'}{u} - \left( W W'' + \left( W' \right)^2 \right) \\
& \text{Ric}(Y_i, Y_j) = 0 \quad \text{ for } i \neq j.
\end{align*}
\end{proposition}

\begin{proof}
First note that, up to the symmetries of the curvature tensor, there are no nonzero terms in Theorem \ref{thm:curvature formulas} of the form $R^{\l}_{i,k,j,k}$ with $i \neq j$.  
This proves that the Ricci tensor is diagonal (the last equation in the Proposition), which is not immediately obvious since the Riemann curvature tensor contains nonzero mixed terms.  

For the first equation in the Proposition, we have in the $i$ odd case
\begin{align*}
&\text{Ric}(Y_i, Y_i) = \sum_{k=1}^{2n} \l \left( R(Y_i, Y_k) Y_i, Y_k \right)  \\
&= R^\l_{i, i+1, i, i+1} + \sum_{k=1, k \neq i, i+1}^{2n-2} \left( R^\l_{i, k, i, k} \right) + R^\l_{i, 2n-1, i, 2n-1} + R^\l_{i, 2n, i, 2n} \\
&= -4 \left( \frac{1 + W^2}{u^2} \right) - (2n-4) \left( \frac{1 + W^2}{u^2} \right) - 2 \frac{W W'}{u} \\
&= -2n \left( \frac{1 + W^2}{u^2} \right) - 2 \frac{W W'}{u}.
\end{align*}
The $i$ even case is identical, and the middle equation is an equally straightforward calculation.
\end{proof}

A metric $\l$ is {\it Einstein} with Einstein constant $C$ if
\begin{equation*}
	\text{Ric}_q(A,B) = C \cdot \l_q(A, B)
\end{equation*}
for all $A, B \in T_q M$ and for all $q \in M$.
Also, recall that the metric $\l$ from \eqref{eqn:lambda V} is a function of $V(u)$ where we set $W^2 = V$ for convenience in our calculations.
The following Theorem describes exactly the functions $V(u)$ for which $\l$ is Einstein with the same Einstein constant as the complex hyperbolic metric.

\begin{theorem}\label{thm:Einstein functions}
The metric $\l$ from \eqref{eqn:lambda} is Einstein with Einstein constant $-2(n+1)$ exactly when
\begin{equation}\label{eqn:V equation}
	V(u) = u^2 - 1 + \frac{\k}{u^{2n}}
\end{equation}
for $\k \in \R$.
\end{theorem}

\begin{remark}
We are intentionally trying to use the same notation as \cite{FP} for readability.
In \cite{FP} the authors use the variable $a$ for the constant in \eqref{eqn:V equation}.
But we have already used $a = \cosh(\s)$ in the previous section, and so we have changed notation to $\k$ to avoid confusion.  
\end{remark}

\begin{proof}[Proof of Theorem \ref{thm:Einstein functions}]
This argument is essentially identical to that in \cite{FP}.  
We need to solve the system of differential equations
\begin{align}
 -2n \left(1 + W^2 \right) - 2 u W W' &= -2(n+1) u^2 \label{eqn:Einstein equation 1} \\
 -(2n+1) W W' - u \left( W W'' + \left( W' \right)^2 \right) &= -2 (n+1) u \label{eqn:Einstein equation 2}.	
\end{align}
But notice that differentiating both sides of \eqref{eqn:Einstein equation 1} with respect to $u$ gives \eqref{eqn:Einstein equation 2}.
So we really just need to solve \eqref{eqn:Einstein equation 1}.

Observing that $V' = 2 W W'$, equation \eqref{eqn:Einstein equation 1} is equivalent to
\begin{equation*}
-2n \left( 1 + V \right) - u V' = -2 (n+1) u^2.
\end{equation*}
Rearranging terms, we can rewrite this equation as
\begin{equation}\label{eqn:V differential equation}
V' + \frac{2n}{u} V = (2n+2) u - (2n) u^{-1}.
\end{equation}
From here, a simple integrating factor of $u^{2n}$ gives the solution
\begin{equation*} 
V = u^2 - 1 + \frac{\k}{u^{2n}}.
\end{equation*}
\end{proof}

\begin{remark}
Notice that the exponent in the denominator of the last term of $V$ in \eqref{eqn:V equation} is different than the exponent in \cite[Proposition 3.2]{FP}.  If $E$ denotes the Einstein constant for the ambient symmetric space, then it appears that the formula for this exponent is $-(E + 2)$.  
\end{remark}

For $\k \in \R$ we use the notation $V_\k$ to denote the function $V$ from \eqref{eqn:V equation} with constant $\k$, and use $\l_\k$ to denote the metric $\l$ with the choice of $V_\k$.
The metric $\l_\k$ is well-defined for all values $\k$ and $u \geq 1$ such that $V_\k(u) > 0$.

\subsection{Cone angles and curvatures of the model Einstein metric} 
In this subsection we consider the metric $\l_\k$ written in polar coordinates on $\R^{2n} = \R^{2n-2} \times \S^1 \times [0, \infty)$. 
When $\k = 0$ the metric $\l_\k = \l_0 = \cm_n$ and thus defines a smooth metric on $\R^{2n}$.  
But when $\k \neq 0$ this metric will have a cone angle about $\R^{2n-2}$.  
The following Lemma, which is mostly the same as from \cite{FP} with the substitution $n-3 = 2m$, quantifies how this cone angle depends on $\k$.  

In what follows we let $u_\k$ denote the largest root of $V_\k$. 
Note that when $u_\k > 0$, the metric $\l_\k$ is well-defined for all $u \in (u_\k, \infty)$.  
We also define $f(u) = (1-u^2) u^{2n}$ and observe that roots of $V_\k$ are equivalent to solutions of the equation $f(u) = \k$.

\begin{lemma}\label{lem:cone angle lemma}
Let
\begin{equation*}
v = \sqrt{\frac{n}{n+1}} \qquad \k_{\text{max}} = \frac{1}{n+1} v^{2n} \qquad c_\k = \frac{u_\k}{2} V_\k'(u_\k) = u_\k^2 - \frac{n \k}{u_\k^{2n}}.
\end{equation*}
\begin{enumerate}
\item The root $u_\k$ is positive if and only if $\k \in (- \infty, \k_{\text{max}}]$.  The map $\gamma : (- \infty, \k_{\text{max}}] \to [v, \infty)$ defined by $\gamma(\k) = u_\k$ is a decreasing homeomorphism.  
\item When $\k \in (- \infty, \k_{\text{max}})$, the metric $\l_\k$ has a cone angle about $\R^{2n-2}$ at $u = u_\k$ with cone angle $2 \pi c_\k$.
\item The map $\b : (- \infty, \k_{\text{max}}] \to [0, \infty)$ defined by $\b(\k) = c_\k$ is a decreasing homeomorphism.  In particular, as $\k$ varies from $0$ to $\k_{\text{max}}$, the cone angle $2 \pi c_\k$ takes every value from $2 \pi$ to $0$ exactly once.
\end{enumerate}
\end{lemma}

Lemma \ref{lem:cone angle lemma} is the last thing that we needed to prove Theorem \ref{thm:main theorem}.

\begin{proof}[Proof of Theorem \ref{thm:main theorem}]
Recall that we need to prove that our model Einstein metric $\l_\k$ is isometric to the K\"{a}hler-Einstein metric $\omega_\a$ whose existence is guaranteed by \cite[Theorem 2.2]{GH}.
In \cite[Theorem 2.9]{GH} the authors argue that $f_\a(r)$, the warping function for the horizontal fiber of $\omega_\a$, must satisfy a specific differential equation.
The initial conditions for this differential equation are $f'_\a(0) = 0$ and $f_\a(0) \in \left( \sqrt{\frac{n}{n+1}},1 \right) = (v,1)$.
With respect to equation \eqref{eqn:lambda V}, the coefficient of the horizontal distribution is $u = \cosh(r)$.  
Also, recall in \eqref{eqn:lambda V} that we made the substitution $V(u) = u^2 - 1 = \sinh^2(r) = (\cosh'(r))^2$.

In the differential equation \eqref{eqn:V differential equation} we set $u = f(r)$ and $V(u) = (f'(r))^2$.
Note that, in terms of \eqref{eqn:V differential equation}
\begin{equation*}
V' = \frac{dV}{du} = \frac{dV}{dr} \cdot \frac{dr}{du} = 2 f'(r) f''(r) \left( \frac{1}{f'(r)} \right) = 2 f''(r).
\end{equation*}
Substituting into \eqref{eqn:V differential equation} gives
\begin{equation*}
2f'' + \frac{2n (f')^2}{f} = (2n+2)f - \frac{2n}{f}
\end{equation*}
which can be rearranged as
\begin{equation*}
\frac{f''}{f} + n \frac{(f')^2}{f^2} + n \frac{1}{f^2} = n+1.
\end{equation*}
This is the differential equation derived in Theorem 2.9 of \cite{GH}.
To see that the solution $V_\k$ of this equation is the same solution as $\omega_\a$, just compare the initial conditions listed above for $f_\a$ with Lemma \ref{lem:cone angle lemma}.
\end{proof}

\begin{proof}[Proof of Lemma \ref{lem:cone angle lemma}]
Again, this proof is very similar to \cite{FP}, but the presence of the $u^2$ in the $d \th^2$ term of \eqref{eqn:lambda} causes enough differences that we verify all details below.  

For (1), just note that $u_\k$ is the largest value of $u$ which satisfies $f(u) = \k$.  
Statement (1) then follows from consideration of the graph of $f(u) - \k$.  

To prove statement (2), fix $\k \in (- \infty, \k_{\text{max}})$.  
To calculate the cone angle of $\l_\k$ we need to consider values of $u$ near $u_\k$.  
One has, for $u$ near $u_\k$,
\begin{equation*}
V_\k(u) = V_\k(u_\k) + V_\k'(u_\k) (u - u_\k) + O \left( (u-u_\k)^2 \right) \approx \frac{2 c_\k}{u_\k} (u-u_\k) .
\end{equation*}
We then, in the metric $\l_\k$, make the substitution
\begin{equation*}
s = \sqrt{\frac{2 u_\k}{c_\k} (u-u_\k)}
\end{equation*}
Direct calculations verify the following:
\begin{align*}
&du^2 = \frac{2 c_\k}{u_\k} (u - u_\k) ds^2 \approx V_\k ds^2 \\
&\frac{c_\k^2 s^2}{u_\k^2} = \frac{2 c_\k}{u_\k} (u - u_\k) \approx V_\k \\
&u^2 = \left( \frac{s^2 c_\k}{2 u_\k} + u_\k \right)^2.
\end{align*}
Substituting for $\l_\k$ in \eqref{eqn:lambda}, we have
\begin{equation}\label{eqn:cone angle equation}
\l_\k \approx \left( \frac{s^2 c_\k}{2 u_\k} + u_\k \right)^2 \cm_{n-1} + \left( \frac{s^2 c_\k}{2 u_\k} + u_\k \right)^2 \frac{c_\k^2 s^2}{u_\k^2} d \th^2 + ds^2.
\end{equation}
Notice that, as $u \to u_\k$, we have $s \to 0$.  
Therefore, as $u \to u_\k$, the leading (linear) coefficient of $d \th^2$ in \eqref{eqn:cone angle equation} approaches $c_\k^2 s^2$.  
Thus, $\l_\k$ has cone angle of $2 \pi c_\k$ about $\R^{2n-2}$.  

To verify (3), recall that $u_\k$ is the largest solution to the equation $f(u) = \k$ or, equivalently, $(1 - u_\k)^2 u_\k^{2n} = \k$.  
This gives
\begin{equation*}
c_\k = u_\k^2 - \frac{n \k}{u_\k^{2n}} = u_\k^2 - \frac{n (1-u_\k^2) u_\k^{2n}}{u_\k^{2n}} = (1+n)u_\k^2 - n.
\end{equation*}
This equation defines an increasing homeomorphism from $[v, \infty) \to [0, \infty)$.  
Composing with $\gamma$ from (1) yields the desired decreasing homeomorphism $\b$.  
For the last claim, just note that when $\k = 0$ we have $u_0 = 1$ and therefore $c_0 = 1$.  
Thus, when $\k = 0$, the cone angle is $2 \pi$ (which is clear since $\l_0 = \cm_n$).  
The claim then follows from the fact that $\b$ is decreasing.
\end{proof}

Notice the immediate corollary of Lemma \ref{lem:cone angle lemma}.

\begin{cor}\label{cor:cone angle}
For any integer $d \geq 2$, there exists a unique $\k = \k_d \in (0, \k_{\text{max}}]$ such that the metric $\l_\k$ has cone angle $2 \pi / d$ about $\C \H^{n-1}$.
\end{cor}

We now aim to find upper and lower bounds for the sectional curvature of $\l_\k$ in order to give an independent proof of \cite[Theorem 2.11]{GH}.  
It turns out that $\l_\k$ also has the same upper and lower curvature bounds as the model Einstein metric in \cite{FP}, but verifying this is more involved due to the nonzero mixed terms in Theorem \ref{thm:curvature formulas}.
Recall that $W^2 = V_\k = u^2 - 1 + \k u^{-2n}$.  
From this, we observe that
\begin{align*}
&2 W W' = V_\k' = 2u - 2 n \k u^{-2n-1} \\
&2 \left( \left( W' \right)^2 + W W'' \right) = V_\k'' = 2 + 2n(2n+1) \k u^{-2n-2}.
\end{align*}
The following Corollary restates the curvature formulas from Theorem \ref{thm:curvature formulas} with the above calculations.

\begin{cor}\label{cor:curvature formulas for lambda kappa}
With respect to the orthonormal basis defined in equations \eqref{eqn:ON basis 1} through \eqref{eqn: ON basis 3} we have, up to the symmetries of the curvature tensor, the following formulas for the nonzero components of the Riemann curvature tensor $R^\k$ of $\l_\k$.
\begin{align}
&R^\k_{i,i+1,i,i+1} = -4 - \frac{4 \k}{u^{2n+2}} \qquad \text{ for } i \text{ odd.} \label{eqn:curvature cor eqn 1} \\
&R^\k_{i,j,i,j} = - 1 - \frac{\k}{u^{2n+2}}. \label{eqn:curvature cor eqn 2}  \\
&R^\k_{i, 2n-1, i, 2n-1} = R^\k_{i,2n,i,2n} = - 1 + \frac{n \k}{u^{2n+2}}. \label{eqn:curvature cor eqn 3}  \\
&R^\k_{2n-1,2n,2n-1,2n} = -4 - \frac{2 n (n-1) \k}{u^{2n+2}}. \label{eqn:curvature cor eqn 4}  \\
&R^\k_{i, i+1, j, j+1} = 2 R^\k_{i, j, i+1, j+1} = -2 R^\k_{i, j+1, i+1, j}= -2 - \frac{2 \k}{u^{2n+2}} \quad \text{ for } i, j \text{ odd}. \label{eqn:curvature cor eqn 5} \\
&R^\k_{i, i+1, 2n-1, 2n} = 2R^\k_{i, 2n-1, i+1, 2n} = -2R^\k_{i, 2n, i+1, 2n-1}= -2 + \frac{2 n \k}{u^{2n+2}}. \label{eqn:curvature cor eqn 6}
\end{align}
In the formulas above we have $1 \leq i, j \leq 2n-2$ and, if $i$ and $j$ appear in the same formula, we assume that $(Y_i, Y_j)$ does not form a holomorphic pair.
We also assume $i$ is odd in \eqref{eqn:curvature cor eqn 6}.
\end{cor}

With these formulas we can now prove the following.

\begin{proposition}\label{prop:lambda negative curvature}
For all $\k \in (0, \k_{\text{max}})$, the metric $\l_\k$ has all sectional curvatures bounded above by a negative constant.  
Moreover, the sectional curvature $K$ of $\l_\k$ satisfies
\begin{equation*}
-4 - \frac{2n(n-1) \k}{u^{2n+2}} \leq K \leq -1 + \frac{n \k}{u^{2n+2}}
\end{equation*}
\end{proposition}

\begin{proof}
First note that, of the curvatures of the coordinate planes in equations \eqref{eqn:curvature cor eqn 1} through \eqref{eqn:curvature cor eqn 4}, the only one that is greater than $-1$ is \eqref{eqn:curvature cor eqn 3}.  
But, using the notation for $\k_{\text{max}}$ and $v$ from Lemma \ref{lem:cone angle lemma}, observe that
\begin{equation*}
-1 + \frac{n \k}{u^{2n+2}} < -1 + \frac{n \k_{\text{max}}}{v^{2n+2}} = -1 + \frac{n}{n+1} \cdot \frac{1}{v^2} = 0.
\end{equation*}
Therefore, $K < 0$ for all coordinate $2$-planes with respect to the orthonormal frame defined in \eqref{eqn:ON basis 1} through \eqref{eqn: ON basis 3}, but we also have to account for the nonzero mixed terms.

In what follows we use the notation from Subsection 2.1. 
Let $E = \C \H^{n-1} \times \S^1$, $q \in E \times (0, \infty)$, and $p = \phi(q)$ where $\phi$ is the orthogonal projection onto the $\C \H^{n-1}$ factor.
Let $\s \subset T_q(E \times (0, \infty))$ be a $2$-plane such that $\s \neq \text{span}(\partial / \partial \th, \partial / \partial u )$.
Our goal is to choose a convenient frame to calculate $K(\s)$.  
By construction, $d \phi (\s)$ is at least $1$-dimensional in $T_p \C \H^{n-1}$. 
Let $A \in \s$ be a unit vector orthogonal to $\partial / \partial u$ which satisfies that $d \phi(A) \neq 0$, and let $\check{X}_1 \in T_p \C \H^{n-1}$ be a unit vector which is parallel to $d \phi(A)$. 
Choose $B \in \s$ so that $(A, B)$ is an orthonormal basis for $\s$.  
Let $\check{X}_2 = J \check{X}_1$, and choose a unit vector $\check{X}_3$ orthogonal to $\text{span}(\check{X}_1, \check{X}_2)$ in such a way that $d \phi(B) \in \text{span}(\check{X}_1, \check{X}_2, \check{X}_3)$.
Extend $(\check{X}_1, \check{X}_2, \check{X}_3)$ to orthogonal vector fields $(X_1, X_2, X_3)$ about $q$ in exactly the same manner as described in Subsection 2.1.  
Finally, let
\begin{equation*}
Y_1 = \frac{1}{u} X_1, \quad Y_2 = \frac{1}{u} X_2, \quad Y_3 = \frac{1}{u} X_3, \quad Y_{5} = \frac{1}{uW} \ddt, \quad Y_{6} = W \frac{\partial}{\partial u}.
\end{equation*}
Then there exist constants $a_1, a_5, b_1, b_2, b_3, b_5, b_6$ such that
\begin{equation*}
A = a_1 Y_1 + a_5 Y_{5} \qquad B = b_1 Y_1 + b_2 Y_2 + b_3 Y_3 + b_5 Y_{5} + b_6 Y_{6}
\end{equation*}
with
\begin{equation*}
a_1^2 + a_5^2 = 1 \qquad \sum b_i^2 = 1 \qquad a_1 b_1 + a_5 b_5 = 0.
\end{equation*}
We then compute
\begin{align*}
&K(\s) = \l_\k(R^\k(A,B)A, B) \\
&= a_1^2 b_2^2 R^\k_{1,2,1,2} + a_1^2 b_3^2 R^\k_{1,3,1,3} + (a_1 b_5 - a_5 b_1)^2 R^\k_{1,5,1,5} + a_1^2 b_6^2 R^\k_{1,6,1,6} + a_5^2 b_2^2 R^\k_{2,5,2,5} \\
& \hskip 6pt + a_5^2 b_3^2 R^\k_{3,5,3,5} + a_5^2 b_6^2 R^\k_{5,6,5,6} + 2a_1 a_5 b_2 b_6 \left( R^\k_{1,2,5,6} - R^\k_{1,6,2,5} \right) \\
&= (4a_1^2 b_2^2 + a_1^2 b_3^2) \left( -1 - \frac{\k}{u^{2n+2}} \right) + ((a_1 b_5 - a_5 b_1)^2 + a_1^2 b_6^2 + a_5^2 b_2^2 + a_5^2 b_3^2) \left( -1 + \frac{n \k}{u^{2n+2}} \right) \\
& \hskip 4pt + a_5^2 b_6^2 \left( -4 - \frac{2n(n-1) \k}{u^{2n+2}} \right) + 6 a_1 a_5 b_2 b_6 \left( -1 + \frac{n \k}{u^{2n+2}} \right) \\
&= (4a_1^2 b_2^2 + a_1^2 b_3^2) \left( -1 - \frac{\k}{u^{2n+2}} \right) + ((a_1 b_5 - a_5 b_1)^2 + a_5^2 b_3^2) \left( -1 + \frac{n \k}{u^{2n+2}} \right) \\
& \hskip 4pt + (a_1 b_6 + a_5 b_2)^2 \left( -1 + \frac{n \k}{u^{2n+2}} \right) + a_5^2 b_6^2 \left( -4 - \frac{2n(n-1) \k}{u^{2n+2}} \right) + 4 a_1 a_5 b_2 b_6 \left( -1 + \frac{n \k}{u^{2n+2}} \right) \\
&< (2a_1^2 b_2^2 + a_1^2 b_3^2) \left( -1 - \frac{\k}{u^{2n+2}} \right) + ((a_1 b_5 - a_5 b_1)^2 + a_5^2 b_3^2) \left( -1 + \frac{n \k}{u^{2n+2}} \right) \\
& \hskip 4pt + (a_1 b_6 + a_5 b_2)^2 \left( -1 + \frac{n \k}{u^{2n+2}} \right) + a_5^2 b_6^2 \left( -2 - \frac{n(n-1) \k}{u^{2n+2}} \right) \\
& \hskip 4pt + 2 (a_1 b_2 + a_5 b_6)^2 \left( -1 + \frac{n \k}{u^{2n+2}} \right). \\
\end{align*}
This last expression is clearly negative.  
To obtain an upper curvature bound, a quick observation shows that we want to choose $a_1 = b_6 = 0$.  
This forces $a_5 = \pm 1$ and, by orthogonality, we must have $b_5 = 0$.  
The last expression above reduces to
\begin{equation*}
\left( -1 + \frac{n \k}{u^{2n+2}} \right) \left( b_1^2 + b_3^2 + b_2^2 \right) = \left( -1 + \frac{n \k}{u^{2n+2}} \right).
\end{equation*}
The lower curvature bound is easily seen to occur when $|a_5| = |b_6| = 1$, which yields
\begin{equation*}
K(\s) = -4 - \frac{2n(n-1) \k}{u^{2n+2}}.
\end{equation*}
\end{proof}

\begin{remark}\label{rmk:extreme curvatures}
The extreme curvature values of $\l_\k$ are functions of both the dimension $n$ and ramification degree $d$.  
The maximum and minimum curvatures both occur at the ramification locus, when $u = u_\k$.  
As $d \to \infty$, we have $u_\k \to v$ and $\k \to \k_{\text{max}}$.  
We saw in the above argument that, in this situation, $(n \k)/u^{2n+2}$ will then approach $1$ as $d \to \infty$.  
The lower curvature bound thus approaches $-4 - 2(n-1) = -2(n+1)$ while the upper curvature bound approaches $0$.  
Therefore, for any $\e > 0$, one can choose the ramification degree $d$ sufficiently large so that the lower curvature bound of $\l_\k$ lies in $(-2(n+1), -2(n+1) + \e)$ while the upper curvature bound lies in $(- \e, 0)$.
Note that these are the same curvature bounds stated in the Introduction of \cite{GH}.

It is also easy to deduce \cite[Theorem 2.4 (4)]{GH} from our previous work.  
Note that $|\l_\k - \cm_n| \leq 2 \k/u^{2n}$.  
Recalling that $u = \cosh(r)$ and $r$ denotes the distance from the branching locus, we see that the model Einstein metric $\l_\k$ approaches $\cm_n$ exponentially in $r$.  
\end{remark}

\section{Complex hyperbolic branched covers and convergence to the K\"{a}hler-Einstein metric}\label{sect:construction of manifolds}
In this section we first show how to construct a sequence of complex hyperbolic branched cover manifolds $(X_k)$ with a specific geometric property (namely, that the normal injectivity radius of the branching locus goes to infinity while the diameter of the branching locus remains fixed).  
We will then discuss how to use our model Einstein metric $\l_\k$ to construct a smooth, approximately Einstein metric $g_k$ on $X_k$ for each $k$.
This metric will be negatively curved for $k$ sufficiently large.  
Finally we will show how to use an inverse function theorem to perturb $g_k$ to obtain a negatively curved Einstein metric $\mathfrak{e}_k$ on $X_k$.  
Since our approximate Einstein metric $g_k$ is isometric to $\omega_{d,R}$ from \cite{GH}, equation (30) from \cite{GH} shows that $\mathfrak{e}_k$ asymptotically approaches the unique K\"{a}hler-Einstein metric on $X_k$ with negative Ricci constant.  

The authors want to emphasize that the work in this section was completed {\it after} the announcement of \cite{GH}.  
This section really just combines known results from the literature, including \cite{GH}, with a small amount of work to check that they apply to our metric $\l_\k$.
But we include full details here for completeness.

\subsection{The sequence of complex hyperbolic branched cover manifolds}
Fix integers $d, n \geq 2$.
The goal of this subsection is to prove the following.

\begin{theorem}\label{thm:construction of manifolds}
There exists a sequence of compact K\"{a}hler manifolds $(M_k)$ of complex dimension $n$ with connected, totally geodesic submanifolds $(N_k)$ which satisfy the following.
\begin{enumerate}
\item The universal cover of $M_k$ is isometric to $\C \H^n$.
\item The submanifold $N_k \subset M_k$ is embedded, has real codimension $2$, and its preimage in the universal cover of $M_k$ is isometric to multiple disjoint copies of $\C \H^{n-1}$. 
\item $N_k$ is isometric to $N_\ell$ for all indices $k, \ell$. 
\item The fundamental class $[N_k] \in H_{2n-2}(M_k, \mathbb{Z})$ is $d$-divisible.
\item The normal injectivity radius $\eta_k$ of $N_k$ in $M_k$ satisfies $\eta_k \geq k$.
\end{enumerate}
\end{theorem}

\begin{remark}
By condition $(4)$, the cyclic $d$-fold branched cover $X_k$ of $M_k$ about $N_k$ is a smooth manifold.  
By $(5)$ the normal injectivity radius of the branching locus of $X_k$ will approach infinity as $k \to \infty$.
The branched cover manifold $X_k$ is K\"{a}hler by \cite{Zheng} and, for $k$ sufficiently large, this is the manifold which admits a negatively curved K\"{a}hler-Einstein metric.
\end{remark}

\begin{proof}[Proof of Theorem \ref{thm:construction of manifolds}]
The construction of the pair $(M_k, N_k)$ is recursive.  
We first show how to construct $(M_0, N_0)$ and then, given $(M_k, N_k)$, we show how to construct $(M_{k+1}, N_{k+1})$ which satisfies $(1)-(5)$ above.  

Using the notation of \cite{ST}, let $\Gamma < \text{PU}(n,1)$ be a cocompact congruence arithmetic lattice of simple type.  
It is known that, after possibly passing to a finite congruence cover, we may assume that the quotient $M = \Gamma \setminus \C \H^n$ contains an embedded submanifold $N$ satisfying $(2)$ above.  
In the event that this submanifold is disconnected, we let $N$ denote one of its components.

By a result of Stover and Toledo \cite[Proposition 5.1]{ST}, there exists a finite cover $(M', N')$ of $(M, N)$ such that $[N'] \in H_{2n-2}(M', \mathbb{Z})$ is $d$-divisible.  
We set $M_0 = M'$ and, if $N'$ is connected, we let $N_0 = N'$.  
If $N'$ contains multiple components, then an application of the Universal Coefficient Theorem \cite[Lemma 3.8]{GH} shows that each individual component of $N'$ must also have fundamental class that is $d$-divisible.  
Choose one such component and call it $N_0$.  

We now assume that the pair $(M_k, N_k)$ satisfying $(1)-(5)$ in Theorem \ref{thm:construction of manifolds} has been constructed, and we show how to construct the pair $(M_{k+1}, N_{k+1})$.  
This argument is analogous to that of \cite[Proposition 3.3]{GH}.  
To start, if the normal injectivity radius $\eta_k$ of $N_k$ is $k+1$ or greater, then we just set $M_{k+1} = M_k$ and $N_{k+1} = N_k$.  
So we may assume that $\eta_k < k+1$.  

Let $\Gamma_k < \text{PU}(n,1)$ be such that $M_k = \Gamma_k \setminus \C \H^n$.  
The preimage of $N_k$ in $\C \H^n$ will generally have many connected components, each isometric to $\C \H^{n-1}$.  
Fix one of these components and call it $V$.  
Let $\Lambda_k = \text{Stab}_{\Gamma_k}(V)$ denote the stabilizer of $V$ in $\Gamma_k$.  
Since $\eta_k < k+1$ there exists a geodesic with endpoints in $N_k$ whose length is at most $2(k+1)$, which meets $N_k$ orthogonally at each endpoint, and which does not admit a homotopy into $N_k$.  
Connecting these endpoints with a geodesic contained in $N_k$ gives a closed path in $M_k$ which corresponds to a nontrivial element $\gamma \in \pi_1(M_k) \cong \Gamma_k$.
Note that, by construction, $\gamma \nin \Lambda_k$.    

By a result of Bergeron \cite{Bergeron}, there exists a finite index subgroup $\Gamma_k' < \Gamma_k$ such that $\Lambda_k < \Gamma_k'$ but $\gamma \nin \Gamma_k'$.  
By \cite{DKV} there exist only finitely many homotopy classes of closed curves in $M_k$ with length less than $2(k+1) + \text{diam}(N_k)$. 
So, after finitely many steps, we obtain a finite index subgroup $\Gamma_{k+1} < \Gamma_k$ with $\Lambda_k < \Gamma_{k+1}$ which satisfies the following.  
Let $M_{k+1} = \Gamma_{k+1} \setminus \C \H^n$, and let $N_{k+1} \subset M_{k+1}$ be equal to $\Lambda_k \setminus V$.  
Then $N_{k+1}$ is connected, is isometric to $N_k$, and has normal injectivity radius $\eta_{k+1} \geq k+1$.  
This verifies conditions $(1) - (3)$ and $(5)$ in the Theorem.

For property $(4)$, first notice that $M_{k+1}$ is a finite cover of $M_k$ since $\Gamma_{k+1} < \Gamma_k$ has finite index.  
Since the fundamental class of $N_k$ is $d$-divisible, by \cite[Lemma 3.7]{GH} we have that the lift of $N_k$ in $M_{k+1}$ has $d$-divisible fundamental class.  
This lift may have multiple connected components, one of which is $N_{k+1}$.  
But \cite[Lemma 3.8]{GH} shows that, if the fundamental class of the preimage of $N_k$ is $d$-divisible, then the fundamental class of each component is $d$-divisible.  
Thus, $N_{k+1}$ satisfies condition $(4)$ of the Theorem.  
\end{proof}

\subsection{The approximate Einstein metric}
In this subsection we consider the manifold pairs $(M_k, N_k)$ constructed in Theorem \ref{thm:construction of manifolds}.  
Fix an integer $d \geq 2$ as in the Theorem, and let $X_k$ denote the $d$-fold cyclic branched cover of $M_k$ about $N_k$.

By Corollary \ref{cor:cone angle} there exists a unique $\k = \k_d \in (0, \k_{\text{max}})$ such that the cone angle of $\l_{\k}$ about some copy of $\C \H^{n-1}$ in $\C \H^n$ is $2 \pi/d$.  
Define a smooth function $\chi: \R \to [0, \infty)$ which satisfies $\chi(u) = 1$ if $u \leq 1/2$ and $\chi(u) = 0$ for $u \geq 1$.  
We then write
\begin{equation*}
	V_k (u) = u^2 - 1 + \frac{\k}{u^{2n}} \chi \left( \frac{u}{\eta_k} \right)
\end{equation*}
and let $g'_k$ denote the metric $\l$ associated with this choice of $V_k$.
Since $V_k = u^2 - 1$ at all points of distance at least $\eta_k$ from $\C \H^{n-1}$, the metric $g'_k$ descends to a well-defined orbifold metric on $M_k$.  
By abuse of notation we again denote this metric $g'_k$, and note that $g'_k$ will be everywhere Riemannian except at $N_k$.  
At each point of $N_k$ the metric $g'_k$ will have a cone angle of $2 \pi / d$.  

Let $\rho_k : X_k \to M_k$ denote the $d$-fold cyclic branched covering map, and equip $X_k$ with the metric $g_k := \rho^*_k(g'_k)$.
By an abuse of notation we also use $N_k$ to denote the branching locus of $X_k$.
By construction the metric $g_k$ is smooth about $N_k$, and $g_k$ is certainly negatively curved and Einstein within the $(1/2) \eta_k$-tube of $N_k$ and outside of the $\eta_k$-tube of $N_k$.  
Since $\eta_k \to \infty$, all derivatives of $\chi ( u / \eta_k )$ approach $0$ as $k \to \infty$.  
Therefore, for $k$ sufficiently large, $g_k$ will be negatively curved everywhere and will be approximately Einstein on the $((1/2) \eta_k, \eta_k)$-annulus about $N_k$.  
This approximation becomes finer as $k \to \infty$. 
We summarize this information in the following Proposition.

\begin{proposition}\label{prop:g_k properties}
There exists an integer $K$ such that, for all $k \geq K$, there exists a smooth Riemannian metric $g_k$ on $X_k$ with the following properties.
\begin{enumerate}
\item There is a constant $c > 0$ such that $\text{sec}(g_k) < -c$, where $\text{sec}$ denotes the sectional curvature.
\item $\text{Ric}(g_k) + (2n+2) g_k$ is nonzero only in the annular neighborhood $(1/2) \eta_k < r < \eta_k$.
\item For any $m \in \mathbb{N}$, there exists a constant $A$ depending on $m$ but not on $k$ such that
	\begin{equation*}
	|| \text{Ric}(g_k) + (2n+2)g_k ||_{C^m} \leq \frac{A}{\left( \cosh \left( \frac{1}{2} \eta_k \right) \right)^{2n+2}}
	\end{equation*}
\item We have
	\begin{equation*}
	\lim_{k \to \infty} \int_{X_k} \bigr| \text{Ric}(g_k) + (2n+2)g_k|^2 d \text{vol}_{g_k} = 0
	\end{equation*}
\end{enumerate}
\end{proposition}

\begin{proof}
Statement $(2)$ is clear by the definition of $\chi$.  
Statement $(1)$ is also clear since the curvature of $g_k'$ is bounded above by a negative constant, and $g_k$ is a $C^2$-approximation of $g_k'$.
Statement $(3)$ is identical to \cite[Proposition 3.1 (1)]{FP} but with replacing $n-3$ by $2n$.

For $(4)$, first note that the integrand is supported on the $((1/2) \eta_k, \eta_k)$-annulus about $N_k$.  
The metric $g_k$ only differs from $\cm_n$ within the normal bundle of $N_k$, and the maximum length of any unit vector in this annulus is
$uW \approx u^2 \leq \cosh^2(\eta_k)$.
Thus, an upper bound for the $g_k$-volume of this region is given by
	\begin{equation*}
	d \cdot \text{vol}(N_k) \cdot \cosh^2(\eta_k) \cdot \eta_k^2.
	\end{equation*}
From Theorem \ref{thm:construction of manifolds} (3) we know that $\text{vol}(N_k)$ is constant.  
Then, using the approximation $\cosh(x) \approx (1/2)e^x$ for $x$ large, statement $(4)$ then follows from statement $(3)$ assuming $n > 1$.  
\end{proof}

\subsection{Perturbing $g_k$ to obtain a negatively curved Einstein metric}
In this subsection we will closely follow the notation and results from \cite[Sections 2, 4]{HJ} and \cite[Section 4]{FP}, both of which originate (at least) from the work of Anderson in \cite{Anderson}.  
The proof of Theorem \ref{thm:main theorem 2} below is identical to the proof of \cite[Theorem 4.3]{HJ}.  
Technically, one would need to replace $n-1$ with $2n+2$ in equations (2.1), (2.3), and the subsequent Lemmas of \cite{HJ}, but the proofs of these Lemmas are independent of this positive constant.  
Alternately, one could simply scale our approximate Einstein metric $g_k'$ appropriately so that its Einstein constant is $-(n-1)$.  
Then, the results of \cite{HJ} apply directly.
The goal of this subsection is to provide the terminology and results necessary from \cite{FP} and \cite{HJ} to justify Theorem \ref{thm:main theorem 2}, with the hope that this subsection can also serve as a gentle introduction for non-experts to the above references.

Let $\Sym$ denote the collection of smooth, symmetric, bilinear forms on $TX_k$ of any signature.  
The collection of Riemannian metrics on $X_k$ forms an open cone within $\Sym$.
One would like to define an operator $E: \Sym \to \Sym$ by
\begin{equation*}
	E(g) = \text{Ric}(g) + (2n+2)g
\end{equation*}
but this operator is not elliptic and thus the inverse function theorem will not generally be applicable.  
The remedy for this issue is called {\it Bianchi guage fixing}.
Given $g \in \Sym$, define $\Phi_g : \Sym \to \Sym$ by
\begin{equation*}
	\Phi_g(h) = \text{Ric}(h) + (2n+2)h + \text{div}^*_h(\beta_g(h))
\end{equation*}
where $\text{div}^*_h$ is the $L^2$-adjoint of the divergence with respect to $g$ and $\beta_g : \Sym \to T^*X$ is the {\it Bianchi operator}.  
We omit the definitions since we will not use them here, but we direct the interested reader to see \cite{FP}, \cite{HJ}, and the references therein.  
An important property of $\Phi_g$ is that, if $|\beta_g|$ is bounded and $\text{Ric}(g) < \lambda g$ for some $\lambda < 0$, then $\Phi_g(h) = 0$ if and only if $h$ satisfies both $\text{Ric}(h) = -(2n+2)h$ and $\beta_g(h) = 0$.  
See \cite[Lemma 2.1]{Anderson} or \cite[Lemma 2.1]{HJ}.
So, at least for metrics with negative Ricci curvature, $\Phi_g$ is sufficient to be able to detect Einstein metrics near $g$.  
But it is generally not necessary as it will miss Einstein metrics $\mathfrak{e}$ which satisfy $\beta_g(\mathfrak{e}) \neq 0$ as discussed in \cite{Anderson}.

The reason why it is advantageous to add $\beta_g$ to $E$ to obtain $\Phi_g$ is that the operator $\Phi_g$ is now elliptic.  
Its linearization (at $g$) can be explicitly calculated as
\begin{equation*}
\mathcal{L}_g(h) := (D \Phi_g)_g(h) = \frac{1}{2} \Delta_Lh + (2n+2)h
\end{equation*}
where $\Delta_L = \nabla^* \nabla h + \text{Ric}(h)$ is the {\it Lichnerowicz Laplacian}.  
A good reference for the notation is \cite[Section 9.3]{Petersen}.  
Koiso \cite{Koiso} proved that, if $g$ is Einstein with negative sectional curvature, then $\mathcal{L}_g$ has a uniform $L^2$-spectral gap.  
Our metric $g_k$ is only approximately Einstein, but an analog to this result was proved in \cite[Lemma 4.4]{FP} (see also the corresponding \cite[Lemma 4.1]{HJ}) to apply to this situation.

All of the results above apply in all dimensions.  
The reason why the negatively curved Einstein metrics from \cite{FP} are only guaranteed to exist in dimension $4$ is twofold.  
The first reason is because the construction of the manifold pairs $(M_k, N_k)$ in \cite{FP} do not give sufficient $L^2$-control of $\Phi_{g_k}(g_k)$.  
Hamenst\"{a}dt and J\"{a}ckel in \cite{HJ} observed that, by using subgroup separability to bound the diameter of the branching locus, one could obtain such $L^2$-control of $\Phi_{g_k}(g_k)$.  
This is encoded in Proposition \ref{prop:g_k properties} (4) above.
The second reason is due to the bound on $\mathcal{L}_g^{-1}$.  
In \cite{FP} Fine and Premoselli are able to find an upper Lipschitz bound for $\mathcal{L}_g^{-1}$ that depends on $k$ \cite[Lemma 4.13]{FP}, whereas Hamenst\"{a}dt and J\"{a}ckel are able to construct a $C^0$-estimate for $\mathcal{L}_g^{-1}$ \cite[Lemma 2.2]{GH} and use this to find such a bound that is {\it independent of $k$} \cite[Proposition 4.2]{GH}.
This \cite[Proposition 4.2]{GH} replaces the estimate needed by Fine and Premoselli in \cite[Theorem 4.15]{FP} and is the key to extending their results to all dimensions.

We are now ready to prove Theorem \ref{thm:main theorem 2}.

\begin{proof}[Proof of Theorem \ref{thm:main theorem 2}]
Consider
\begin{equation*}
	\Phi_{g_k} : C^2(\Sym) \to C^0(\Sym)
\end{equation*}
where $C^\ell(\Sym)$ denotes $\Sym$ with the $C^\ell$-topology.  
By \cite[Proposition 4.2]{GH} there exists $\e > 0$, independent of $k$, such that $\Phi_{g_k}$ surjects onto the $\e$-neighborhood of $\Phi_{g_k}(g_k)$, denoted $B(\Phi_{g_k}(g_k), \e)$.  
By Proposition \ref{prop:g_k properties} (4) we have that $||g_k||_{L^2} \to 0$, and consequently $||\Phi_{g_k}(g_k)||_{C^0} \to 0$.  
Thus there exists $K \in \mathbb{N}$ such that, for all $k \geq K$, $0 \in B(\Phi_{g_k}(g_k), \e)$.
Therefore, there exists $\mathfrak{e}_k \in C^2(\Sym)$ such that $\Phi_{g_k}(\mathfrak{e}_k) = 0$.  
It is known that $\Phi_{g_k}$ is Lipschitz with constant independent of $k$ (see \cite[Lemma 4.1]{FP} and \cite[Proposition 4.2]{HJ}).  
Hence, by possibly choosing a larger $K$ if necessary, we have that $\mathfrak{e}_k$ is positive-definite and negatively curved.  
The metric $\mathfrak{e}_k$ is thus the desired negatively curved Einstein metric.
Moreover, by letting $\e \to 0$ as $k \to \infty$, one can ensure that
\begin{equation*}
\lim_{k \to \infty} ||g_k - \mathfrak{e}_k||_{C^2} = 0
\end{equation*}
\end{proof}

\begin{remark}\label{rmk:einstein approximation}
From the equation above, we see that our model Einstein metric $C^2$-approximates the negatively curved Einstein metric $\mathfrak{e}_k$, and this approximation gets closer as $k \to \infty$.  
By equation $(30)$ in \cite{GH} we see that $g_k$ also $C^2$-approximates the negatively curved K\"{a}hler-Einstein metric $\omega_k$.  
It is unclear to the authors if, eventually, one must have $\mathfrak{e}_k = \omega_k$ for all $k$, or if this sequence of manifolds supports a sequence of negatively curved Einstein metrics that $C^2$-converge to $\omega_k$.  
\end{remark}



\vskip 20pt

\bibliographystyle{plain}
\bibliography{references.bib}

\begin{thebibliography}{10}

\bibitem{Anderson}
Michael~T. Anderson.
\newblock Dehn filling and {Einstein} metrics in higher dimensions.
\newblock {\em J. Differ. Geom.}, 73(2):219--261, 2006.

\bibitem{Aubin}
Thierry Aubin.
\newblock Equations du type {Monge}-{Amp{\`e}re} sur les vari{\'e}t{\'e}s {K{\"a}hleriennes} compactes.
\newblock {\em Bull. Sci. Math., II. S{\'e}r.}, 102:63--95, 1978.

\bibitem{belegradek2012complex}
Igor Belegradek.
\newblock Complex hyperbolic hyperplane complements.
\newblock {\em Math. Ann.}, 353(2):545--579, 2012.

\bibitem{Bergeron}
Nicolas Bergeron.
\newblock The first {Betti} number and the {Laplace} spectrum of certain hyperbolic manifolds.
\newblock {\em Enseign. Math. (2)}, 46(1-2):109--137, 2000.

\bibitem{Bland}
John~S. Bland.
\newblock The {Einstein}-{K{\"a}hler} metric on {{\(\{| z| ^ 2+| w| ^{2p}<1\}\)}}.
\newblock {\em Mich. Math. J.}, 33:209--220, 1986.

\bibitem{CY}
Shiu-Yuen Cheng and Shing-Tung Yau.
\newblock On the existence of a complete {K{\"a}hler} metric on non-compact complex manifolds and the regularity of {Fefferman}'s equation.
\newblock {\em Commun. Pure Appl. Math.}, 33:507--544, 1980.

\bibitem{DKV}
J.~J. Duistermaat, J.~A.~C. Kolk, and V.~S. Varadarajan.
\newblock Spectra of compact locally symmetric manifolds of negative curvature.
\newblock {\em Invent. Math.}, 52:27--93, 1979.

\bibitem{FP}
Joel Fine and Bruno Premoselli.
\newblock Examples of compact {E}instein four-manifolds with negative curvature.
\newblock {\em J. Amer. Math. Soc.}, 33(4):991--1038, 2020.

\bibitem{GT}
M.~Gromov and W.~Thurston.
\newblock Pinching constants for hyperbolic manifolds.
\newblock {\em Invent. Math.}, 89(1):1--12, 1987.

\bibitem{GH}
Henri Guenancia and Ursula Hamenst\"{a}dt.
\newblock K\"{a}hler-einstein metrics of negative curvature.
\newblock {\em preprint, arXiv: 2503.02838}, 2025.

\bibitem{HJ}
Ursula Hamenst\"{a}dt and Frieder J\"{a}ckel.
\newblock Negatively curved einstein metrics on gromov-thurston manifolds.
\newblock {\em preprint, arXiv: 2411.12956}, 2025.

\bibitem{Hirzebruch-RamifiedCoverings}
Friedrich Hirzebruch.
\newblock The signature of ramified coverings.
\newblock Global {Analysis}, {Papers} in {Honor} of {K}. {Kodaira} 253-265 (1969)., 1969.

\bibitem{Jackel}
Frieder J\"{a}ckel.
\newblock Effective stability of negatively curved einstein metrics in dimensions at most 12.
\newblock {\em preprint, arXiv: 2502.16692}, 2025.

\bibitem{Koiso}
Norihito Koiso.
\newblock Non-deformability of {Einstein} metrics.
\newblock {\em Osaka J. Math.}, 15:419--433, 1978.

\bibitem{minemyer2018real}
Barry Minemyer.
\newblock Real hyperbolic hyperplane complements in the complex hyperbolic plane.
\newblock {\em Adv. Math.}, 338:1038--1076, 2018.

\bibitem{MinemyerKahler}
Barry Minemyer.
\newblock K\"{a}hler manifolds with an almost $1/4$-pinched metric.
\newblock {\em preprint, arXiv: 2307.15550v2}, 2023.

\bibitem{Minemyercurvatureformulas}
Barry Minemyer.
\newblock Warped product metrics on hyperbolic and complex hyperbolic manifolds.
\newblock {\em To appear in Algebr. Geom. Topol.}, 2025.

\bibitem{Petersen}
Peter Petersen.
\newblock {\em Riemannian geometry}, volume 171 of {\em Grad. Texts Math.}
\newblock Cham: Springer, 3rd edition edition, 2016.

\bibitem{Premoselli}
Bruno Premoselli.
\newblock Negatively curved {Einstein} metrics on ramified covers of closed four-dimensional hyperbolic manifolds.
\newblock In {\em Actes de s\'eminaire de th\'eorie spectrale et g\'eom\'etrie. Ann\'ee 2017--2019}, pages 129--161. St. Martin d'H{\`e}res: Universit{\'e} de Grenoble I, Institut Fourier, 2019.

\bibitem{ST}
Matthew Stover and Domingo Toledo.
\newblock Residual finiteness for central extensions of lattices in {${\rm PU}(n, 1)$} and negatively curved projective varieties.
\newblock {\em Pure Appl. Math. Q.}, 18(4):1771--1797, 2022.

\bibitem{Viro}
O.~Ya. Viro.
\newblock The signature of branched covering.
\newblock {\em Mat. Zametki}, 36(4):549--557, 1984.

\bibitem{YauA}
Shing-Tung Yau.
\newblock A general {Schwarz} lemma for {K{\"a}hler} manifolds.
\newblock {\em Am. J. Math.}, 100:197--203, 1978.

\bibitem{Yau}
Shing-Tung Yau.
\newblock On the {Ricci} curvature of a compact {K{\"a}hler} manifold and the complex {Monge}-{Amp{\`e}re} equation. {I}.
\newblock In {\em Complex geometry from Riemann to K\"ahler-Einstein and Calabi-Yau}, pages 285--345. Somerville, MA: International Press; Beijing: Higher Education Press, 2018.

\bibitem{Zheng}
Fangyang Zheng.
\newblock Examples of non-positively curved {K}\"ahler manifolds.
\newblock {\em Comm. Anal. Geom.}, 4(1-2):129--160, 1996.

\end{thebibliography}

\end{document}